\documentclass{article} 
\usepackage[T1]{fontenc}
\usepackage{indentfirst}
\usepackage{amsmath}
\usepackage{amssymb}
\usepackage{amsthm}
\usepackage{cancel}
\usepackage{dsfont}
\usepackage{stmaryrd}
\SetSymbolFont{stmry}{bold}{U}{stmry}{m}{n}
\usepackage{hyperref}
\usepackage{empheq}
\usepackage{cleveref}
\usepackage{bm}
\usepackage{enumitem}
\usepackage[pdftex]{graphicx}
\usepackage{mathtools}
\usepackage{pgf, tikz}
\usepackage{appendix}

\title{Continuous dependence of the pressure field with respect to endpoints for ideal incompressible fluids}
\author{Aymeric \textsc{Baradat} \footnote{CMLS, \'Ecole Polytechnique and \'Ecole Normale Sup\'erieure, France. \newline
E-mail: \emph{aymeric.baradat@polytechnique.edu}}}
\date{}

\theoremstyle{plain}
\newtheorem{Thm}{Theorem}

\newtheorem{Lem}[Thm]{Lemma}
\theoremstyle{definition}
\newtheorem{Def}[Thm]{Definition}

\theoremstyle{remark}
\newtheorem{Rem}[Thm]{Remark}

\newcommand{\R}{\mathbb{R}}
\newcommand{\Z}{\mathbb{Z}}
\newcommand{\T}{\mathbb{T}}

\newcommand{\N}{\mathbb{N}}
\newcommand{\A}{\mathcal{A}}
\newcommand{\C}{\mathcal{C}}
\newcommand{\PM}{\mathcal{P}}
\newcommand{\Pb}{\mathsf{Pb}}

\newcommand{\cg}{\langle}
\newcommand{\cd}{\rangle}
\newcommand{\1}{\mathds{1}}
\newcommand{\eps}{\varepsilon}
\newcommand\pf{_{\#}}

\DeclareMathOperator{\Lip}{Lip}
\DeclareMathOperator{\Div}{div}

\DeclareMathOperator{\Sdiff}{SDiff}

\DeclareMathOperator{\Diam}{Diam}

\DeclareMathOperator{\D}{d \!}
\DeclareMathOperator{\DD}{d}
\DeclareMathOperator{\DDD}{D}
\DeclareMathOperator{\DMK}{d_{MK}}

\DeclareMathOperator{\Id}{Id}
\DeclareMathOperator{\loc}{loc}
 
\begin{document} 
\maketitle
\begin{abstract}
In the Brenier variational model for perfect fluids, the datum is the joint law of the initial and final positions of the particles. In this paper, we show that both the optimal action and the pressure field are H\"older continuous with respect to this datum metrized in Monge-Kantorovic distance. 
\end{abstract}
\section*{Introduction}
The movement of an inviscid incompressible fluid without any external force in a domain $D$ (here, $D$ will be the $d$-dimensional torus $\T^d := \R^d / \Z^d$) is usually described by a time dependent vector field $ v = v(t,x)$ satisfying the Euler equations:
\[\left\{
\begin{aligned}
&\partial_t v(t,x) + v(t,x) \cdot \nabla v(t,x) = - \nabla p(t,x), &(t,x) \in \R_+ \times \T^d, \\
&\Div v(t,x) = 0,  &(t,x) \in \R_+ \times \T^d. 
\end{aligned}\right.
\]
It is now well understood since the works of Arnold \cite{arn66} or more recently \cite{arn99}, that these equations are the formal geodesic equations on the formal infinite dimensional Lie group of well-oriented measure preserving diffeomorphisms $\Sdiff(\T^d)$, in which the (right-invariant) action of a curve $g = (g(t,x))$ is defined by
\[
\frac{1}{2} \int_0^1 \hspace{-5pt} \int |\partial_t g(t,x) |^2 \D x \D t.
\]
However Shnirelman showed in \cite{shn87} that the study of the geometry of this group leads to difficulties: for example, in dimension 2, some diffeomorphisms cannot be connected to the identity map by a curve of finite length, and in dimension 3, there are diffeomorphisms that cannot be connected by optimal curves (the limiting object of the minimizing procedure cannot be a curve of diffeomorphisms).

That is why in \cite{bre89}, Brenier introduced a relaxed model having all the good properties of a variational problem. Let us describe it. Take a positive Radon measure $\gamma$ on the product space $\T^d \times \T^d$ whose both marginals are the Lebesgue measure $\lambda$ (in the sequel, such a measure will be called bistochastic measure on $\T^d$). An admissible generalized flow for the problem $\Pb(\gamma)$ is by definition a generalized flow $\eta$, that is to say a Borel measure on $\C := C^0([0,1]; \T^d)$, such that:
\begin{itemize}
\item the joint law between the initial time and the final time is $\gamma$, namely for all measurable nonnegative function $\varphi$ on $\T^d \times \T^d$,
\begin{equation}
\label{boundarycondition}
\int \varphi (\omega(0), \omega(1)) \D \eta (\omega) = \int \varphi(x,y) \D \gamma(x,y)
\end{equation}
(we say that $\eta$ satisfies the endpoints condition $\gamma$),
\item the marginal at every time $t \in [0,1]$ of $\eta$ is the Lebesgue measure, \textit{i.e} for all measurable nonnegative function $\varphi$ on $\T^d$,
\begin{equation}
\label{incompressibility}
\int \varphi (\omega(t)) \D \eta (\omega) = \int \varphi(x) \D x,
\end{equation}
\item the flow has a finite action:
\begin{equation}
\label{finiteactionflow}
\mathcal{A}(\eta) := \int \frac{1}{2} \int_0^1 |\dot{\omega}(t)|^2 \D t \D \eta(\omega) < + \infty,
\end{equation}
where here and in all this text, the action of a curve is set to $+ \infty$ if the curve is not absolutely continuous.
\end{itemize}

The relaxed solutions to incompressible Euler equations with endpoints condition $\gamma$ are the solutions to the minimization problem $\Pb(\gamma)$ consisting in finding the admissible flows with endpoints condition $\gamma$ that minimize $\mathcal{A}$ in the class of such flows. With an abuse of notations, we will write $\A(\gamma)$ the action of the optimal flows in $\Pb(\gamma)$. Remark that the action is proper and lower semi-continuous with respect to the topology of narrow convergence of generalized flows, so that the existence of solutions is equivalent to the existence of admissible flows. It was shown in \cite{bre89} that in the $d$-dimensional torus (or in the $d$-dimensional cube), such flows always exist, but are in general non-unique. It is shown furthermore that the optimal action is bounded uniformly in $\gamma$. We call this property the finite diameter property of the Brenier model. In other terms
\begin{equation}
\label{finitediameter}
\Diam(d) := \sup_{\gamma \mbox{\scriptsize{ bistochastic on }} \T^d} \sqrt{ \A(\gamma) } < + \infty.
\end{equation}
 This property will be useful in the sequel. 

Following Brenier, a full theory was developed to study these objects and it seems to the author that \cite{amb09} is a good overview of what is known on the topic. One of the main results of this theory already seen in \cite{bre93} is the existence of a unique scalar pressure field $p \in \mathcal{D}'(]0,1[ \times \T^d)$ which depends only on $\gamma$ and which satisfies for all generalized solution $\eta$ with endpoint condition $\gamma$ and for all $\alpha \in \mathcal{D}(]0,1[ \times \T^d)$ the equation
\begin{equation}
\label{defpressure}
\cg p, \Div \alpha \cd_{\mathcal{D}', \, \mathcal{D}} = - \int \hspace{-5pt} \int_0^1 \dot{\omega}(t) \cdot \left( \frac{\D}{\D t} \alpha(t, \omega(t))  \right) \D t \D \eta(\omega).
\end{equation}

This pressure field is interpreted as the Lagrange multiplier associated to the incompressibility constraint in the minimization problem in the following sense. If we take any generalized flow $H$ with endpoints condition $\gamma$, and if we call $R(t,\bullet) = 1 + r(t,\bullet)$ the marginal at time $t$ of $H$, then as soon as $r$ is sufficiently regular, 
\begin{equation}
\label{pmultlag}
\A(\gamma)+ \cg p, R - 1 \cd_{\mathcal{D}', \, \mathcal{D}} = \A(\gamma)+ \cg p, r \cd_{\mathcal{D}', \, \mathcal{D}} \leq \A(H). 
\end{equation}

In \cite{brenier1992motion}, Brenier gave an even more intuitive interpretation of the pressure field, which was proven rigorously in \cite{amb09} by Ambrosio and Figalli: the solutions of the minimisation problem only charge trajectories that satisfy Lagrange least action principle, with a potential given by the pressure field. More explicitly, if $\eta$ is a solution to the minimization problem, and if $p$ is its pressure field, then for all $0<s \leq t < 1$, $\eta$-almost all curve $\omega$, minimizes the functional
\begin{equation}
\label{lagrange}
w \mapsto \int_s^t \left\{ \frac{1}{2}|\dot{w}(u)|^2 - p(u, w(u)) \right\} \D u
\end{equation}
in the class of curves $w$ that share their locations with $\omega$ at times $s$ and $t$. In other terms, the problem generates a certain potential $p$, and the solutions follow the paths of classical mechanics with respect to this potential. This has been done thanks to an important regularity result proved in \cite{ambfig08}, improving a previous result given in \cite{bre93}, namely that $p$ is of regularity $L^2_{\loc}(]0,1[, BV(\T^d))$. More precisely, for all $0< \tau \leq 1/4$, there is $M$ only depending on $d$ and $\tau$ (in particular not depending on $\gamma$) such that
\begin{equation}
\label{pl2bv}
\| p \|_{L^2([\tau, 1-\tau]; BV(\T^d))} \leq M.
\end{equation} 
In particular, this result shows that $p$ is an $L^p$ function for some $p>1$, which is surprisingly sufficient to deal with the quantity in \eqref{lagrange} for $\eta$-almost all curve. 

Let us mention that more regularity is expected. In \cite{Brenier2013}, Brenier conjectured that the pressure field should be semi-concave in space (the second order derivatives of $p$ should be measure-valued, and not the first order derivatives). He also gave an example of solution for which the pressure is semi-concave but not $C^1$. The question of regularity is crucial since for instance, in dimension $d=1$, a uniqueness result has been proved in \cite{ber09} under the condition that $p$ is of regularity $C^{\infty}$.

Here, we will discuss a linked but different topic. The purpose of the present paper is to study the stability of the relaxed solutions to incompressible Euler equations with respect to the endpoints condition. More explicitly, we will show that the action and the pressure field are H\"older continuous with respect to the endpoints condition metrized with the Monge-Kantorovitch distance. In a forthcoming article \cite{Baradat2018kinetic}, we will use very different techniques to bound from above the H\"older exponents appearing in these estimates, showing in particular that the results presented here cannot be much improved.

\paragraph{Notations.} We call $\DD_2$ the Euclidean distance on $(\T^d)^2$ and $\DMK$ the Monge-Kantorovich distance of exponent 2 on the set of Borel probability measures on $(\T^d)^2$. Remark that because of the finite diameter of the torus,
\begin{equation}
\label{finitediameterbistochastic}
\sup_{\mu, \nu \in \PM((\T^d)^2)} \DMK(\mu,\nu) < +\infty.
\end{equation}

For all $t \in [0,1]$, $e_t$ will be the map from $\C$ to $T^d$ defined for all $\omega \in \C$ by 
\begin{equation}
\label{defet}
e_t ( \omega) := \omega(t). 
\end{equation}

We will denote by $AC^2([0,1];\T^d)$ the set of absolutely continuous curves $\omega$ from $[0,1]$ to $\T^d$ such that
\begin{equation}
\label{finiteactioncurve}
A(\omega) := \frac{1}{2} \int_0^1 |\dot{\omega}(t) |^2 \D t < + \infty.
\end{equation}
The number $A(\omega)$ is called the action of the curve $\omega$ and is set to $+ \infty$ if $\omega$ is not absolutely continuous. Thus, if $\eta$ is a generalized flow, its action designed by \eqref{finiteactionflow} also reads
\begin{equation}
\label{redefactionflow}
\A(\eta) = \int A(\omega) \D \eta (\omega).
\end{equation}

If $\mathcal{X}$ and $\mathcal{Y}$ are two measurable sets, $m$ is a measure on $\mathcal{X}$ and $f$ is a measurable map from $\mathcal{X}$ to $\mathcal{Y}$, we will denote by $f \pf m$ the push-forward of $m$ by $f$, that is the measure on $\mathcal{Y}$ defined by the property
\begin{equation}
\label{defpushforward}
\forall \varphi: \mathcal{Y} \to \R_+ \mbox{ measurable}, \quad\int \varphi(y) \D f\pf m (y) = \int \varphi(f(x)) \D m(x). 
\end{equation}
In particular, using \eqref{defet}, \eqref{boundarycondition} can be rewritten "$(e_0,e_1) \pf \eta = \gamma$", and \eqref{incompressibility} can be rewritten "for all $t \in [0,1]$, $e_t {}\pf \eta = \lambda$".

A functional space will be of particular interest. This is the space $E$ of continuous functions $f = (f(t,x))$ satisfying the properties:
\begin{itemize}
 \item for all $t \in [0,1]$, $f(t, \bullet )$ is Lipschitz and 
 \begin{equation}
 \label{finitesuplipf}
 \sup_t \Lip f(t, \bullet) < + \infty;
 \end{equation}
 \item for all $x \in \T^d$, $f(\bullet , x) \in AC^2([0,1])$ and the temporal derivative $\partial_t f$ which is punctually defined for almost all $t \in [0,1]$ for all almost all $x \in \T^d$ satisfies
 \begin{equation}
 \label{AC2Linfty}
 \int_0^1 \| \partial_t f(t, \bullet) \|_{L^{\infty}(\T^d)}^2 \D t < + \infty.
 \end{equation}
\end{itemize}
If $f \in E$, we call 
\begin{equation}
\label{defN}
N(f) :=  \sup_t \Lip f(t, \bullet) +  \left( \int_0^1 \| \partial_t f(t, \bullet) \|_{L^{\infty}(\T^d)}^2 \D t \right)^{1/2}.
\end{equation}
Remark that when restricted to functions having zero mean or to functions cancelling for $t$ equal to $0$ or $1$, this is a norm. If there is $\tau \in ]0,1/4[$ such that $f \in E$ cancels for $t \in [0,\tau]\cup[1-\tau , 1]$, we write $f \in E_{\tau}$. With an abuse of notations, we keep the same notations if $f$ has its values in $\R^d$ or $\T^d$.

Let us now give an outline of the paper. It is divided in three parts. 

First, we will show that the optimal action is H\"older continuous with respect to the endpoint conditions. More precisely, we will show the following theorem.
\begin{Thm}
\label{continuityaction}
There exists $M >0$ only depending on $d$ such that for all bistochastic measures $\mu$ and $\nu$ on $\T^d$,
\[
\A(\nu) \leq \A(\mu) + M \DMK(\mu, \nu)^{1/(d+3)}.
\]
As the role of $\mu$ and $\nu$ are symmetric, $\A$ is $1/(d+3)$-H\"older continuous. 

For example, if $d=3$, the action is $1/6$-H\"older continuous.
\end{Thm}
The proof of this theorem will be widely inspired by the pioneering work \cite{shn94}. In that paper, given $\gamma$ a bistochastic measure, the author presents a technique to compare the optimal action in $\Pb(\gamma)$ with the action of the (compressible) generalized flow in which the particles move along straight lines. In our proof, given $\mu$ and $\nu$ two bistochastic measures and $\eta$ an optimal flow in $\Pb(\mu)$, we will build from $\eta$ a (compressible) generalized flow having $\nu$ as endpoints condition and which is close to $\eta$. Then, we will use the same technique to compare the optimal action in $\Pb(\nu)$ with the action of this flow. In fact, we will prove a more general lemma authorizing other marginals than Lebesgue, as needed to prove the stability of the pressure field. 

In a second time, we will show in an explicit example that Theorem \ref{continuityaction} is no longer true if we replace Brenier's model by its generalization discussed by Ambrosio and Figalli in \cite{amb09}.

Finally, we will exploit the first section to show some H\"older continuity property for the pressure field. The result is the following.
\begin{Thm}
\label{stabilitypressure}
Take $\tau \in ]0,1/4[$. There exists $M$ depending only on $d$ and $\tau$ such that for all bistochastic measures $\mu$ and $\nu$ on $\T^d$, 
\[
\sup_{\{ \xi \in E_{\tau}, \, N(\xi) \leq 1 \}} \big| \big\cg \nabla p_{\nu} - \nabla p_{\mu}, \xi \big\cd \big| \leq M \DMK(\mu, \nu)^{1/[2+2(d+1)(d+2)]}.
\]
The gradient of the pressure field is $1/[2+2(d+1)(d+2)]$-H\"older continuous with respect to the endpoints condition (when measured in the dual of $E_{\tau}$).

For example, in dimension $3$, the gradient of the pressure field is $1/42$-H\"older continuous with respect to the endpoints condition.
\end{Thm}
  In the proofs, we will use the letter $M$ to denote a big constant depending only on the dimension (and on $\tau$ in the last section), which will be likely to grow from line to line.
\section{H\"older continuity of the optimal action}
As announced, this section is devoted to the proof of a generalized version of Theorem \ref{continuityaction}. Let us begin by describing it. We define a more general class of problems indexed as before by bistochastic measures, but also by prescribed densities. Remark that if $\eta$ is a generalized flow, then its marginal at time $t$, that we call $\eta_t$ is a probability measure on $\T^d$. Furthermore, the dominated convergence theorem let us deduce that the curve $t \mapsto \eta_t$ is continuous for the topology of narrow convergence (we will write $(\eta_t) \in C^0([0,1]; \PM(\T^d))$). The problems that we consider are the following ones.

\begin{Def}
\label{pbmurho}
Let $\gamma$ be bistochastic and $\rho = (\rho_t) \in C^0([0,1]; \PM(\T^d))$, we say that $\eta$ is an admissible flow for $\Pb(\gamma, \rho)$ if it satisfies \eqref{boundarycondition}, \eqref{finiteactionflow} and if for all $t \in [0,1]$, the marginal at time $t$ of $\eta$ is $\rho_t$. Then, the solutions of $\Pb(\gamma, \rho)$ are the minimizers of the action in the class of admissible flows. Still with abuses of notations, we call $\mathcal{A}(\gamma, \rho)$ the optimal action in $\Pb(\gamma, \rho)$. 
\end{Def}
\begin{Rem}
A necessary condition for solutions to exist is $\rho_0 = \rho_1 = \lambda$. If so, we write $\rho - \lambda \in C^0_0([0,1]; \PM(\T^d)- \lambda)$. This can be generalized without difficulty to other initial or final marginals, but this would add some useless details.
\end{Rem}
As before, the existence of minimizers is equivalent to the existence of admissible flows, and some weak results of existence can be found in Theorem 6.2 of \cite{amb09}, but in fact we will not really be interested in this question, and we will just write $\A(\gamma, \rho) = + \infty$ when there is no solution.
 
From now on, we chose $\psi$ a smooth nonnegative scalar function on $\R^d$ whose support is included in $[-1/4, 1/4]^d$ and of integral equal to one. Then for all $0 < \eps \leq 1$, and $v \in \R^d$, we define
\begin{equation}
\label{defpsieps}
\psi^{\eps}(v) := \frac{1}{\eps^d} \psi\left(\frac{v}{\eps}\right).
\end{equation}  
Of course these functions can be transported to functions on the torus by the natural injection from $[-1/4, 1/4]^d$ to $\T^d$ and we still call the resulting functions $(\psi^{\eps})$.

For reasons that will become clear, if $\rho - \lambda \in C^0_0([0,1]; \PM(\T^d)- \lambda)$ and if $0<\eps\leq1/4$ we call $\rho^{\eps}$ the function of $t \in [0,1]$ and $z \in \T^d$ defined by
\begin{equation}
\label{defrhoeps}
\rho^{\eps}(t,z) := \left\{ \begin{aligned} &1 &&\mbox{if } t \in [0, \eps],\\
& \int \psi^{\eps} (z-x) \D \rho_s(x) \mbox{ with } s = \frac{t-\eps}{1 - 2 \eps} &&\mbox{if } t \in [\eps, 1-\eps],\\
&1 && \mbox{if } t \in [1-\eps, 1].
\end{aligned}\right.
\end{equation}

Finally, if $m$ is a measure on $\T^d$ and $a>0$, we say that $m \geq a$ if $m - a \cdot \lambda$ is a positive measure, and if $\rho \in C^0([0,1]; \PM(\T^d))$, we say that $\rho$ is greater than $a$ if it is the case at all times $t\in [0,1]$. When regularizing the paths of measures, we will not denote differently the regularized paths of measures and their densities with respect to the Lebesgue measure.

The result that we will show is the following. We recall that $\Diam(D)$ is defined by \eqref{finitediameter} and that $N$ is defined by \eqref{defN}.
\begin{Lem}
\label{generalcontinuityaction}
Take $\mu$ a bistochastic measure on $\T^d$ and a prescribed density $\rho$ such that $\rho - \lambda = (\rho_t - \lambda )\! \in \! C^0_0([0,1]; \PM(\T^d)- \lambda)$ and $\rho$ is greater than $3/4$. Then there exist two constants $C >0$ and $M>0$ only depending on $d$ such that for all bistochastic measures $\nu$, and for all $0 <  \eps \leq 1/4$, if
\begin{gather}
\label{conditionlemma1}
\sqrt{\A(\mu, \rho)} \leq 2 \Diam(d),\\
\label{conditionlemma2} C (1 + N(\rho^{\eps}))\frac{\DMK(\mu, \nu)}{\eps^{d+2}} \leq 1/4,
\end{gather}
then
\begin{equation}
\label{estimateAnurhoeps}
\A(\nu, \rho^{\eps}) \leq \A(\mu, \rho) + M \bigg( \eps + \big\{1 + N( \rho^{\eps}) \big\} \frac{\DMK(\mu, \nu) }{\eps^{d+2}} \bigg).
\end{equation}

In particular, if $\rho = \lambda$, then \eqref{conditionlemma1} is always true by \eqref{finitediameter}, $N(\lambda) = 0$, and as soon as
\begin{equation}
\label{conditionlemma2bis} C\frac{\DMK(\mu, \nu)}{\eps^{d+2}} \leq 1/4,
\end{equation}
then
\begin{equation}
\label{estimateAnu}
\A(\nu) \leq \A(\mu) + M \bigg( \eps + \frac{\DMK(\mu, \nu) }{\eps^{d+2}} \bigg).
\end{equation}
\end{Lem}
\begin{Rem}
It is quite easy to see if $\A(\mu, \rho) < + \infty$, then $N(\rho^{\eps}) < + \infty$. In fact, it is always true (use the expression of $\rho^{\eps}$ derived in the following proof to prove it) that there exists $K>0$ only depending on the dimension such that
\begin{equation}
\label{Nrhoeps}
N(\rho^{\eps}) \leq \frac{K \left( 1 + \sqrt{\A(\mu, \rho)}\right)}{\eps^{d+1}}
\end{equation}
 We keep this version because it makes it possible to use the regularity of $\rho$ instead of the one obtained by the regularization procedure if needed.  

If we do not suppose that $\sqrt{\A(\mu, \rho)} < 2 \Diam(d)$, it is possible to prove that the result \eqref{estimateAnurhoeps} is still true replacing $M$ by $M(1 + \A(\mu, \rho))$. But this proof requires to be more cautious and to handle separately the time and space derivatives in lemma \ref{dacmos}.
\end{Rem}
Theorem \ref{continuityaction} follows from this lemma.
\begin{proof}[Proof of theorem \ref{continuityaction}]
We chose 
\[
\eps = \DMK(\mu, \nu)^{1/(d+3)}.
\]
As soon as 
\[
\DMK(\mu, \nu) \leq \frac{1}{(4C)^{d+3}},
\]
equation \eqref{conditionlemma2bis} is satisfied and the result is true. The global H\"older continuity is implied by the local one because of the finite diameter property \eqref{finitediameterbistochastic}. 
\end{proof}
\begin{proof}[Proof of lemma \ref{generalcontinuityaction}]
Take $\mu$, $\nu$ and $\rho$ as in the statement of the lemma. Take $\eta$ a solution to $\Pb(\mu, \rho)$ as designed in Definition \ref{pbmurho}. Also take $\Gamma$ an optimal plan between $\mu$ and $\nu$ (in the classical sense of quadratic optimal transport, $\Gamma$ is a probability measure on $(\T^d)^4$). 

Let us explain heuristically the idea of the proof. Take $(x,y,X,Y) \in (\T^d)^4$. Let us imagine that the particles moving from $x$ to $y$ in $\eta$ follow the path $\omega$ with probability $\D \eta (\omega)$. We can modify $\omega$ by defining
\[
\xi(t) := \omega(t) + (1-t) (X-x) + t (Y-y), 
\] 
to get a path from $X$ to $Y$. This transformation is illustrated at Figure \ref{omegatoxi}. 
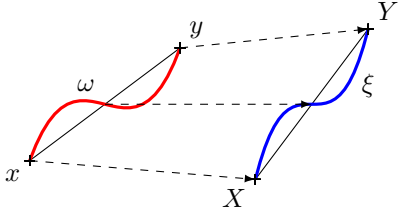
\begin{figure}
\centering
\begin{tikzpicture}
\draw (-1,-0.75) -- (1,0.75);
\draw (2,-1) -- (3.5,1);
\draw[very thick, red][domain=-1:1] plot(\x, {\x*\x*\x - \x/4});
\draw[very thick, blue][domain=2:3.5] plot(\x, {(4*\x-11)*(4*\x-11)*(4*\x-11)/27});
\draw[->, >=latex, dashed] (-1,-0.75) -- (2,-1);
\draw[->, >=latex, dashed] (1,0.75) -- (3.5,1);
\draw[->, >=latex, dashed] (0,0) -- (2.75,0);
\draw plot[mark = +] (-1,-0.75) node[below left]{$x$};
\draw plot[mark = +] (1,0.75) node[above right]{$y$};
\draw plot[mark = +] (2,-1) node[below left]{$X$};
\draw plot[mark = +] (3.5,1) node[above right]{$Y$};
\draw (-0.25,0.25) node{$\omega$};
\draw (3.5,0.25) node{$\xi$};
\end{tikzpicture}
\caption{\label{omegatoxi} Illustration of the operation that gives $\xi$ from $\omega$. The curve $\xi$ is obtained by adding to $\omega$ the only affine function for the endpoints of $\xi$ to be $X$ and $Y$.}
\end{figure}
But by definition,
\begin{equation}
\label{dMKmunu}
\DMK(\mu, \nu)^2 = \int \DD_2\big( (x,y), (X,Y)\big)^2 \D \Gamma(x,y,X,Y).
\end{equation}
(We recall that $\DD_2$ is the Euclidean distance on $(\T^d)^2$.) So as $\mu$ and $\nu$ are supposed to be close, $(x,y)$ and $(X,Y)$ are expected to be close for a large amount of $(x,y,X,Y)$ according to $\Gamma$. For such $(x,y,X,Y)$, the path $\xi$ will be a slight modification of $\omega$. We can then define a plan by charging $\xi$ with the mass $\D \Gamma(x,y,X,Y) \D \eta(\omega)$ and do this transformation for every path. We may then end up with a flow which has $\nu$ as endpoints condition and which is close to $\eta$. We can then straighten it to make it have the density $\rho$, and we get an admissible plan for $\Pb(\nu, \rho)$ whose action should not be very larger than $\A(\eta)$. In fact, to straighten the flow, we need to regularize it and that is why we will diffuse a bit the particles, giving rise to the parameter $\eps$ of the statement. This strategy consisting in regularizing a generalized flow in order to straighten its density is borrowed from the proof of Shnirelman in \cite{shn94}. However, if Shnirelman builds a straightening map by hands, we prefer to use the famous result of Dacorogna and Moser presented in \cite{dac90}. We give a simple version of this result in Lemma \ref{dacmos}.

We now start the rigorous proof. We fix $\nu$ a bistochastic measure satisfying \eqref{conditionlemma1} and a parameters $0 < \eps \leq 1/4$ . We divide the reasoning in several steps during which we will progressively modify $\eta$ to end up with an admissible flow for $\Pb(\nu,\rho^{\eps})$. At each step, we will derive an upper bound for the action of the flow that would have been built.

In the three first steps, $C$ will be a dimensional constant which may grow from line to line. It will be fixed in the end of step three.
\paragraph{Step one: change of endpoints condition.} First, as announced, we change the endpoint conditions of the solution $\eta$ to $\Pb(\mu, \rho)$ defined in Definition \ref{pbmurho}. The set $\C$ of continuous curves on $\T^d$ endowed with the supremum norm is a Polish space and the evaluation map $(e_0, e_1)$ defined by \ref{defet} is measurable, so we can use the disintegration theorem to define for $\mu$-almost all $(x,y) \in (\T^d)^2$
\begin{equation}
\label{defetaxy}
\eta^{x,y} := \eta( \bullet | \omega(0) = x, \, \omega(1) = y).
\end{equation}
For all $(x,y,X,Y) \in (\T^d)^4$, define for all curves $\omega \in \C$ the curve $T_1[x,y,X,Y](\omega)$ whose position at time $t \in [0,1]$ is
\begin{equation}
\label{defT1}
T_1[x,y,X,Y](\omega)(t) = \omega(t) + (1-t)(X-x) + t(Y-y),
\end{equation}
as already shown at Figure \ref{omegatoxi}. This curve moves $\eta^{x,y}$-almost surely from $X$ to $Y$. Then we introduce the generalized flow
\begin{equation}
\label{defeta1}
\eta_{\Gamma,T_1} := \int T_1[x,y,X,Y] \pf \eta^{x,y} \D \Gamma(x,y,X,Y),
\end{equation}
using notation \eqref{defpushforward}. (We recall that $\Gamma$ is an optimal plan between $\mu$ and $\nu$.)

Let us check the endpoints condition of $\eta_{\Gamma,T_1}$. For any $\alpha: (\T^d)^2 \to \R$ which is smooth,
\begin{align*}
\int \alpha(\omega(0), \omega(1)) \D \eta_{\Gamma,T_1}(\omega) &= \int \left( \int \alpha (X, Y ) \D \eta^{x,y}(\omega) \right) \D \Gamma(x,y,X,Y) \\
&=\int \alpha(X,Y) \D \Gamma(x,y,X,Y) \\
&= \int \alpha(X,Y) \D \nu(X,Y).
\end{align*}
The endpoints condition of $\eta_{\Gamma,T_1}$ is $\nu$. Let us now compute the action of $\eta_{\Gamma,T_1}$. For all $(x,y,X,Y) \in (\T^d)^4$ and $\omega \in \C$, using the triangle inequality in $L^2$ and \eqref{finiteactioncurve}, we get
\begin{align*}
A\big(T_1[x,y,X,Y](\omega) \big) &= \frac{1}{2} \int_0^1 |\dot{\omega}(t) - (X-x) + (Y-y)|^2 \D t \\
&=\frac{1}{2} \Big\| \dot{\omega}- (X-x) + (Y-y) \Big\|_{L^2([0,1])}^2 \\
&\leq  \frac{1}{2} \Big\{ \|  \dot{\omega} \|_{L^2([0,1])} + \|  (X-x) - (Y-y) \|_{L^2([0,1])} \Big\}^2\\
&=\Big\{\frac{1}{\sqrt{2}}  \|  \dot{\omega} \|_{L^2([0,1])} + \frac{1}{\sqrt{2}} \|  (X-x) - (Y-y) \|_{L^2([0,1])} \Big\}^2\\
&\leq  \Big\{ \sqrt{A(\omega)} + \frac{1}{\sqrt{2}}\DD_2\big( (x,y), (X,Y) \big) \Big\}^2.
\end{align*}
Integrating with respect to $\eta^{x,y}$ and then $\Gamma$, using \eqref{redefactionflow}, \eqref{dMKmunu} and \eqref{defeta1}, we get
\begin{align*}
\A(\eta_{\Gamma,T_1}) &\leq \left\| \sqrt{A(\omega)} + \frac{1}{\sqrt{2}} \DD_2\big( (x,y), (X,Y) \big) \right\|_{L^2(\D \eta^{x,y}(\omega) \D \Gamma(x,y,X,Y))}^2 \\
&\leq \left\{ \Big\| \sqrt{A(\omega)} \Big\|_{L^2(\D\eta(\omega))} + \frac{1}{\sqrt{2}} \Big\| \DD_2\big( (x,y), (X,Y) \big) \Big\|_{L^2(\D \Gamma(x,y,X,Y))}\right\}^2\\
&\leq \left( \sqrt{\A(\eta)} + \frac{\DMK(\mu, \nu)}{\sqrt{2}} \right)^2\\
&= \A(\eta) + \sqrt{2 \A(\eta)} \DMK(\mu, \nu) + \frac{1}{2} \DMK(\mu, \nu)^2.
\end{align*}
Thus, using \eqref{finitediameterbistochastic} and \eqref{conditionlemma1}, we find $C>0$ such that,
\begin{equation}
\label{estim1}
\A(\eta_{\Gamma,T_1}) - \A (\eta) \leq C \DMK(\mu, \nu).
\end{equation}
\paragraph{Step two: regularization.} To regularize the marginals in space, we define for all $v \in \R^d$ and for any curve $\omega \in \C$ the curve $T_2[v](\omega)$ whose position at time $t$ is
\begin{equation}
\label{defT2}
T_2[v](\omega)(t) = \left\{ \begin{aligned}
&\omega(0) + \frac{v}{\eps}t && \mbox{if } t \in [0,\eps],\\
&\omega(s) + v  \quad \mbox{with } s = \frac{t-\eps}{1-2\eps} && \mbox{if } t\in [\eps, 1 - \eps],\\
&\omega(1) + \frac{1-t}{\eps} v && \mbox{if } t \in [1-\eps, 1]. 
\end{aligned} \right.
\end{equation}
Then we define 
\begin{equation}
\label{defeta2}
\eta_{\Gamma,T_1,T_2} := \int \Big( T_2[v] \pf \eta_{\Gamma,T_1} \Big) \psi^{\eps}(v) \D v,
\end{equation}
where $\psi^{\eps}$ is the cutoff function \eqref{defpsieps} used in Definition \eqref{defrhoeps} of $\rho^{\eps}$. As for all $v$, $T_2[v]$ fixes the endpoints of the curves, the endpoints condition of $\eta_{\Gamma,T_1,T_2}$ is still $\nu$. Let us compute its action. For all $v \in \R^d$ and $\omega \in \C$, 
\[
A(T_2[v](\omega)) = \frac{|v|^2}{\eps} + \frac{1}{2(1-2\eps)}\int_0^1 |\dot{\omega}(t)|^2 \D t.
\]
Integrating with respect to $\eta_{\Gamma,T_1}$ and then $\psi^{\eps}(v) \D v$, and using \eqref{defrhoeps} and $\eps \leq 1/4$, we get
\begin{align*}  
\A(\eta_{\Gamma,T_1,T_2}) &= \int \frac{|v|^2}{\eps} \psi^{\eps}(v) \D v + \frac{1}{1-2 \eps}\A(\eta_{\Gamma,T_1})\\
&\leq \eps \int |w|^2\psi(w) \D w + (1 + 4 \eps) \A(\eta_{\Gamma,T_1}).
\end{align*}
Subsequently, using \eqref{estim1} and \eqref{finitediameterbistochastic}, we easily find $C$ such that
\begin{equation}
\label{estim2}
\A(\eta_{\Gamma,T_1,T_2}) - \A(\eta_{\Gamma,T_1}) \leq C\eps.
\end{equation}
\paragraph{Step three: study of the density of $\boldsymbol{\eta_{\Gamma,T_1,T_2}}$.} We define $Q= (Q(t,z))$ as the "density" of $\eta_{\Gamma,T_1,T_2}$. More explicitly, for all $t$, $Q(t, \bullet) := e_t {}\pf \eta_{\Gamma,T_1,T_2}$ so that by \eqref{defpushforward}, for all $t$ and all test function $\alpha$ on $\T^d$,
\begin{equation}
\label{defQ}
\int \alpha (\omega(t)) \D \eta_{\Gamma, T_1, T_2}(\omega) =  \int \alpha(z) Q(t,z) \D z.
\end{equation}

 In this paragraph, we will bound the quantity $N(Q -\rho^{\eps})$. We recall that $N$ is defined in \eqref{defN}. In the following computations, we will use the notations $\D \eta^{x,y}$, $\D \Gamma$, $s$ and $\xi(s)$ as abbreviations for $\D \eta^{x,y}(\omega)$, $\D \Gamma(x,y,X,Y)$, $(t-\eps)/(1-2 \eps)$ and $\omega(s) + (1-s)(X-x) + s(Y-y)$ respectively. 

First of all, if $t \in [0,\eps] \cup [1-\eps, 1]$, $Q(t, \bullet) \equiv 1$ (like $\rho^{\eps}$). Indeed, if $t \in [0, \eps]$ and if $\alpha$ is a test function on $\T^d$, by \eqref{defT1}, \eqref{defeta1}, \eqref{defT2} \eqref{defeta2} and \eqref{defQ},
\begin{align*}
\int \alpha(z) Q(t,z) \D z &= \int \hspace{-5pt} \int \psi^{\eps}( v ) \alpha\left(X + \frac{t}{\eps} v \right) \D v \D \Gamma \\
&=  \int \hspace{-5pt} \int \psi^{\eps}( v) \alpha\left(X + \frac{t}{\eps} v \right) \D X \D v \\
&= \int \psi^{\eps}(v) \left( \int \alpha\left(X + \frac{t}{\eps} v \right) \D X \right) \D v \\
&= \int \alpha(z) \D z.
\end{align*}
The cases $t \in [1-\eps, 1]$ are treated similarly.

If $t \in [\eps, 1 - \eps]$, we have by \eqref{defT1}, \eqref{defeta1}, \eqref{defT2} \eqref{defeta2} and \eqref{defQ},
\begin{align*}
\int \alpha(z) Q(t,z) \D z &= \int \hspace{-5pt} \int \psi^{\eps}(v )\left\{ \int \alpha\big( \xi(s) + v  \big) \D \eta^{x,y}\right\}\D v \D \Gamma\\
&=  \int \hspace{-5pt} \int \left\{\int \psi^{\eps}(v) \alpha\big( \xi(s) + v  \big) \D v\right\} \D \eta^{x,y} \D \Gamma\\
&=  \int \hspace{-5pt} \int \left\{ \int \psi^{\eps}\Big( z - \xi(s)\Big) \alpha( z ) \D z\right\}\D \eta^{x,y} \D \Gamma,\\
\end{align*} 
which implies that for all $z \in \T^d$, 
\[
 Q(t,z) = \int \hspace{-5pt} \int \psi^{\eps}\Big( z - \xi(s) \Big)\D \eta^{x,y} \D \Gamma.
\]
Then, consider the definition \eqref{defrhoeps} of $\rho^{\eps}$. On the one hand, $\rho$ is the density of $\eta$, that is for all $t$, $\rho_t = e_t {}\pf \eta$. On the other hand, both the endpoints condition of $\eta$ and the first marginal of $\Gamma$ is $\mu$. So in particular, by \eqref{defetaxy}, for all test function $\gamma$ on $\C$,
\begin{align*}
\int \gamma(\omega) \D \eta(\omega) &= \int \hspace{-5pt}  \int \gamma(\omega) \D \eta^{x,y}(\omega) \D \mu(x,y) \\
&= \int \hspace{-5pt}  \int \gamma(\omega) \D \eta^{x,y}(\omega) \D \Gamma(x,y,X,Y)\\
&= \int \hspace{-5pt}  \int \gamma(\omega) \D \eta^{x,y} \D \Gamma \quad \mbox{(with our notations)}.
\end{align*}
Consequently, if $t \in [\eps, 1-\eps]$, we have
\begin{align*}
\rho^{\eps}(t,z) &=\int \psi^{\eps}(z-x) \D \rho_s(x) \\
&= \int \psi^{\eps}\Big( z - \omega(s) \Big)\D \eta(\omega) \\
&=\int \hspace{-5pt} \int \psi^{\eps}\Big( z - \omega(s) \Big)\D \eta^{x,y} \D \Gamma.
\end{align*}

As a consequence, $Q$ and $\rho^{\eps}$ are equal if $t \in [0,\eps]\cup[1-\eps, 1]$, and elsewise, for all $(t, z) \in [\eps,1-\eps] \times D$, we obtain
\begin{equation}
\label{Q-rho}
Q(t,z) - \rho^{\eps}(t,z) = \int \hspace{-5pt} \int \Big\{ \psi^{\eps}\Big( z - \xi(s) \Big) - \psi^{\eps}\Big( z - \omega(s) \Big) \Big\} \D \eta^{x,y} \D \Gamma. 
\end{equation}
If we derive this expression with respect to space, we get
\[
\nabla Q(t,z) - \nabla \rho^{\eps}(t,z) = \int \hspace{-5pt} \int \Big\{\nabla \psi^{\eps}\Big( z - \xi(s) \Big) - \nabla\psi^{\eps}\Big( z - \omega(s) \Big) \Big\} \D \eta^{x,y} \D \Gamma 
\]
Thus, 
\begin{align*}
|\nabla Q(t,z) &- \nabla \rho^{\eps}(t,z)|\\
&\leq \| \DD^2 \psi^{\eps}\|_{\infty} \int \hspace{-5pt} \int |\xi(s) - \omega(s)|\D \eta^{x,y} \D \Gamma \\
&=\frac{\| \DD^2 \psi\|_{\infty}}{\eps^{d+2}} \int  |(1-s)(X-x) + s(Y-y)| \D \Gamma(x,y,X,Y)\\
&\leq \frac{\| \DD^2 \psi\|_{\infty}}{\eps^{d+2}} \int \sqrt{|X-x|^2 + |Y-y|^2} \D \Gamma(x,y,X,Y)\\
&=\leq \frac{\| \DD^2 \psi\|_{\infty}}{\eps^{d+2}} \int \DD_2 ((x,y), (X,Y)) \D \Gamma(x,y,X,Y)\\
&\leq \frac{\| \DD^2 \psi\|_{\infty}}{\eps^{d+2}} \DMK(\mu, \nu) \quad \mbox{(by the Cauchy-Schwarz inequality)}\\
&\leq \frac{C}{\eps^{d+2}} \DMK(\mu, \nu).
\end{align*}
If now we derive \eqref{Q-rho} with respect to time, we get
\begin{align*}
\partial_t Q(t,z) - \partial_t \rho^{\eps}(t,z)  &= \int \hspace{-5pt} \int \Big\{  \nabla \psi^{\eps}( z - \omega(s) ) - \nabla \psi^{\eps}( z - \xi(s)) \Big\} \! \cdot  \dot{\omega}(s) \D \eta^{x,y} \D \Gamma \\
& \quad +\int \hspace{-5pt} \int \nabla \psi^{\eps}( z - \xi(s) )\cdot \left(\dot{\omega}(s) - \dot{\xi}(s)\right) \D \eta^{x,y} \D \Gamma.
\end{align*}
On the one hand, for almost all $t$,
\begin{align*}
\bigg\| \int \hspace{-5pt} \int & \Big\{   \nabla \psi^{\eps}( \bullet - \xi(s)) - \nabla \psi^{\eps}\left( \bullet - \omega(s) \right) \Big\} \cdot \dot{\omega}(s) \D \eta^{x,y} \D \Gamma \bigg\|_{L^{\infty}(\T^d)} \\
 &\leq \| \DD^2 \psi^{\eps} \|_{\infty} \int \hspace{-5pt} \int \left| \xi(s) - \omega(s) \right| \left| \dot{\omega}(s) \right| \D \eta^{x,y} \D \Gamma\\
 &\leq \frac{\|\DD^2 \psi\|_{\infty}}{\eps^{d+2}} \int \hspace{-5pt} \int \left| (1-s)(X-x) + s(Y-y) \right| \left| \dot{\omega}(s) \right| \D \eta^{x,y} \D \Gamma\\
 &\leq \frac{\|\DD^2 \psi\|_{\infty}}{\eps^{d+2}}\left( \int |\dot{\omega}(s)|^2 \D \eta \right)^{1/2} \left( \int \DD_2\big( (x,y),(X,Y) \big)^2 \D \Gamma \right)^{1/2}.
\end{align*}
We take the $L^2$ norm of this expression in space:
\begin{align*}
\bigg(\int_0^1 \bigg\| \int \hspace{-5pt} \int & \Big\{   \nabla \psi^{\eps}( \bullet - \xi(s)) - \nabla \psi^{\eps}\left( \bullet - \omega(s) \right) \Big\} \cdot \dot{\omega}(s) \D \eta^{x,y} \D \Gamma \bigg\|_{L^{\infty}(\T^d)}^2 \D t\bigg)^{1/2} \\
&\leq \frac{C}{\eps^{d+2}} \DMK(\mu, \nu)
\end{align*}
On the other hand, for almost all $t$,
\begin{align*}
\bigg\| \int \hspace{-5pt} \int \nabla &\psi^{\eps}(\bullet - \xi(s) )\cdot \left(\dot{\xi}(s) - \dot{\omega}(s)\right) \D \eta^{x,y} \D \Gamma\bigg\|_{L^{\infty}(\T^d)} 
\\
&= \bigg\| \int \hspace{-5pt} \int \nabla \psi^{\eps}( \bullet - \xi(s))\cdot \big((X-x) -(Y-y)\big) \D \eta^{x,y} \D \Gamma\bigg\|_{L^{\infty}(\T^d)}\\
&\leq \|\nabla \psi^{\eps} \|_{\infty} \DMK(\mu, \nu) \leq \frac{ C}{\eps^{d + 1}}\DMK(\mu, \nu).
\end{align*}
So we get:
\[
\left(\int_0^1 \| \partial_t Q - \partial_t \rho^{\eps}\|_{L^{\infty}(\T^d)}^2 \D t \right)^{1/2} \leq \frac{C}{\eps^{d+2}} \DMK(\mu, \nu).
\]
Gathering these two estimates, and using the definition of $N$ \eqref{defN}, we obtain
\begin{equation}
\label{Q-rhoE}
N(Q - \rho^{\eps} ) \leq \frac{C}{\eps^{d+2}} \DMK(\mu, \nu).
\end{equation}

From now on, the constant $C$ will be fixed and we suppose that \eqref{conditionlemma2} is satisfied. We will call $M$ a new constant which will be "sufficiently large".
\paragraph{Step four: Dacorogna and Moser's lemma.} Now we want to use the following lemma which is a simple version of the main result in \cite{dac90}. We give an elementary proof in the appendix.
\begin{Lem}
\label{dacmos}
Let $f$ and $g$ be in the space $E$ (defined in \eqref{finitesuplipf}, \eqref{AC2Linfty}), greater than $1/2$, and such that for all $t \in [0,1]$,
\[
\int f(t,x) \D x = \int g(t,x) \D x = 1. 
\]
There exists $\Psi \in E$ and $M>0$ only depending on the dimension such that
\begin{gather}
\label{transport}
\forall t \in [0,1], \quad \Psi(t, \bullet) \pf( f(t, \bullet) \cdot \lambda ) = g(t,\bullet) \cdot \lambda,\\
\label{samedensitynomove} \forall t \in [0,1], \quad \mbox{if } f(t,\bullet) = g(t, \bullet), \mbox{ then } \Psi(t, \bullet) = \Id,\\
\label{psiclosetoid}
N(\Psi - \Id) \leq \exp\Big[ M \big(1 + \max\big\{N(f), N(g)\big\}\big) N(g-f) \Big] - 1.
\end{gather}
In particular, for all $\Upsilon>0$ there is $M>0$ only depending on $d$ and $\Upsilon$ such that as soon as
\[
\big(1 + \max\big\{N(f), N(g)\big\}\big) N(g-f) \leq \Upsilon,
\]
then
\[
N(\Psi - \Id) \leq M \big(1 + \max\big\{N(f), N(g)\big\}\big) N(g-f).
\]
\end{Lem}

The density $\rho$ is greater than $3/4$, thus equation \eqref{defrhoeps} shows that it is also the case for $\rho^{\eps}$. In particular, $\rho^{\eps}$ is greater than $1/2$. Furthermore, as a consequence of \eqref{conditionlemma2},
\[
\frac{C}{\eps^{d+2}} \DMK(\mu, \nu) \leq \frac{1}{4}.
\]
So by \eqref{Q-rhoE}, $Q$ defined in \eqref{defQ} is also greater than $1/2$. 

In addition, still by \eqref{Q-rhoE},
\[
1 + \max (N(\rho^{\eps}) , N(Q)) \leq \frac{5}{4} (1 + N(\rho^{\eps})).
\]

So we can apply lemma \ref{dacmos} with $f = \rho^{\eps}$, $g = Q$ and $\Upsilon = 1/4$, and find $\Psi = (\Psi(t,z))$ such that for all $t \in [0,1]$, and for $M>0$ sufficiently large, using \eqref{conditionlemma2},
\begin{gather}
\label{psiQrho} \Psi(t,\bullet) \pf Q(t, \bullet) = \rho^{\eps},\\
\label{Npsi-Id} N(\Psi - \Id ) \leq M (1 + N(\rho^{\eps}))\frac{\DMK(\mu, \nu)}{\eps^{d+2}}.
\end{gather}

Moreover, if $t \in [0,\eps] \cup [1-\eps, 1]$ and $z \in D$, by \eqref{samedensitynomove}, $\Psi(t,z) = z$. Now to a curve $\omega \in \C$, we associate the curve $T_3(\omega) := \Big(t \mapsto \Psi(t, \omega(t))\Big)$ and we define
\[
\eta_{\Gamma, T_1, T_2, T_3} := T_3 {}\pf \eta_{\Gamma,T_1,T_2}.
\]
It is clear that $\eta_{\Gamma,T_1,T_2,T_3}$ has $\nu$ as endpoints condition. Let us show that its density is $\rho^{\eps}$. If $t \in [0,\eps] \cup [1-\eps, 1]$, this is obvious. Else, if $t \in [\eps, 1-\eps]$ and if $\alpha$ is a test function on $\T^d$,
\begin{align*}
\int \alpha(\omega(t)) \D \eta_{\Gamma,T_1,T_2,T_3}(\omega) &= \int \alpha\Big[\Psi(t, \omega(t))\Big] \D \eta_{\Gamma,T_1,T_2} (\omega) \\
&= \int \alpha\big(\Psi(t, z)\big) Q(t,z) \D z\quad \mbox{(by definition \eqref{defQ} of }Q\mbox{)}\\
&=\int \alpha(Z) \rho^{\eps}(t,Z) \D Z \quad \mbox{(by \eqref{psiQrho})}.
\end{align*}
It remains to compute the action of $\eta_{\Gamma,T_1,T_2,T_3}$. If $\omega$ is a curve in $\C$, using the triangle inequality in $L^2$ as in step one,
\begin{align*}
A( T_3(\omega) ) &= \frac{1}{2} \int_0^1 \left| \frac{\D}{\D t} \Psi(t, \omega(t))\right|^2 \D t\\
&= \frac{1}{2} \int_0^1 \left| \dot{\omega}(t) + \partial_t \Psi(t, \omega(t)) + \Big(\D \Psi(t, \omega(t)) - \Id\Big)\cdot \dot{\omega}(t)  \right|^2 \D t\\
&\leq \left(\sqrt{A(\omega)} +  N(\Psi-\Id) \Big( 1 + \sqrt{A(\omega)} \Big) \right)^2.
\end{align*}
As a consequence, still by the same technique
\begin{align*}
\A(\eta_{\Gamma,T_1,T_2,T_3}) &= \int A (T_3(\omega)) \D \eta_{\Gamma,T_1,T_2}(\omega) \\
&\leq \int \left(\sqrt{A(\omega)} +  N(\Psi - \Id) \Big( 1 + \sqrt{A(\omega)} \Big) \right)^2 \D \eta_{\Gamma,T_1,T_2}(\omega)\\
&\leq \left( \sqrt{\A(\eta_{\Gamma,T_1,T_2})} + N(\Psi - \Id) \left( 1 + \sqrt{\A(\eta_{\Gamma,T_1,T_2}}\right) \right)^2\\
&\leq \A(\eta_{\Gamma,T_1,T_2}) + C N(\psi - \Id)\\
&\leq \A(\eta_{\Gamma,T_1,T_2}) + CM (1 + N(\rho^{\eps}))\frac{\DMK(\mu, \nu)}{\eps^{d+2}} \quad \mbox{(by \eqref{Npsi-Id})}.
\end{align*}
So taking $M\leftarrow CM$,
\begin{equation}
\label{estim3}
\A(\eta_{\Gamma,T_1,T_2,T_3}) - \A(\eta_{\Gamma,T_1,T_2}) \leq M (1 + N(\rho^{\eps}))\frac{\DMK(\mu, \nu)}{\eps^{d+2}}.
\end{equation}
We get the result of Lemma \ref{generalcontinuityaction}, namely \eqref{estimateAnurhoeps}, by summing \eqref{estim1}, \eqref{estim2} and \eqref{estim3}.
\end{proof}
\section{Discontinuity of the optimal action in the extended model}
In \cite{amb09}, Ambrosio and Figalli actually worked in a slightly more general context. In their model, the endpoints condition is prescribed by two bistochastic measures $\mu$ and $\nu$, and a flow is a Borel measure on $\overline{\C} := \T^d \times \C$. We call $\pi$ and $W$ the canonical projections $\overline{\C} \to \T^d$, $\overline{\C} \to \C$ and if $t \in [0,1]$, we call $e_t$ the map $(a, \omega)\in \overline{\C} \mapsto \omega(t) \in \T^d$. The flow $H$ is admissible for the problem $\overline{\Pb}(\mu, \nu)$ if the push-forward $W \pf H$ satisfies \eqref{incompressibility} and \eqref{finiteactionflow}, and if $(\pi, e_0) \pf H = \mu$ and $(\pi, e_1) \pf H = \nu$. Under these constraints, the first marginal $\pi \pf H$ is necessarily the Lebesgue measure, and it can be useful to use the disintegration lemma to write $H$ as $\lambda \otimes \eta^a$, which means that $(\eta^a)_{a \in \T^d}$ is a measurable family of Borel measures on $\C$ satisfying for all test function $\alpha$ on $\overline{\C}$,
\[
\int_{\overline{\C}} \alpha (a, \omega) \D H(a, \omega) = \int_{\T^d} \left(\int_{\C} \alpha(a, \omega) \D \eta^a(\omega) \right) \D a.
\]
 This time, the action of a flow $H$ is define by 
 \[
 \overline{\A}(H) := \int \frac{1}{2} \int_0^1 |\dot{\omega}(t)|^2 \D t \D H(a,\omega) = \A(W \pf H),  
 \]
 and once again, a solution to $\overline{\Pb}(\mu, \nu)$ is an admissible flow with minimal action, and we call this optimal action $\overline{\A}(\mu, \nu)$. 

The particles are not only indexed by their initial and final positions anymore, but also by an additional variable in $\T^d$. Notice that the choice of $\T^d$ is quite arbitrary, and we could have chosen any other Polish space. As discussed in \cite{brenier2003extended}, this is the natural way to generalize Brenier's model to obtain something closer to the usual Lagrangian formulation in classical mechanics, with an initial and a final state. This problem inherits a satisfactory structure: for example, it provides a metric on the set of bistochastic measures, which is in addition invariant under a measure preserving change of indices. This property is the analogue in this context of the right-invariance property on the formal Lie group of measure preserving diffeomorphisms described in \cite{arn99}. 

If $\mu = (\Id, \Id) \pf \lambda=: \Lambda$, $\overline{\Pb}(\mu, \nu)$ is exactly $\Pb(\nu)$. Indeed, if $\eta$ is admissible for $\Pb(\nu)$, then $(e_0, \Id) \pf \eta$ is admissible for $\overline{\Pb}(\Lambda, \nu)$ and has the same action as $\eta$, and reciprocally, if $H$ is admissible for $\overline{\Pb}(\Lambda, \nu)$, $W \pf H$ is admissible for $\Pb(\nu)$ and has the same action as $H$. (Just remark that for $H$-almost all $(a, \omega)$, $\omega(0) = a$.) The same kind of arguments shows that if $\nu = \Lambda $, $\overline{\Pb}(\mu, \nu)$ is exactly $\Pb(\mu^{\dagger})$ where $\mu^{\dagger}$ is obtained from $\mu$ by exchanging the two coordinates in $\T^d \times \T^d$. In particular, by theorem \ref{continuityaction}, $\overline{\A}(\Lambda, \bullet)$ and $\overline{\A}(\bullet , \Lambda)$ are H\"older continuous.

On the other side, if $\mu \neq \Lambda$, $\overline{\A}(\mu, \bullet)$ is not even continuous in general, and there is a corresponding statement with final states instead of initial states.   

We will give an example in dimension 1 and in $[0,1]$ instead of $\T^1$ to be more visual, but the same example works in $\T^1$. We define $\mu_{\infty}$ by the formula
\begin{align*}
\forall \alpha\in &C^0([0,1]^2),\\
 &\int \alpha(x,y) \D \mu_{\infty}(x,y) = \frac{1}{2}\int_0^1 \left\{ \alpha\left( x, \frac{x}{2} \right) + \alpha \left( x, \frac{1}{2} + \frac{x}{2}\right)\right\}\D x,
\end{align*}
and for all $n \in \N^*$, $\mu_n$ is defined by the formula
\begin{align*}
\forall \alpha& \in C^0([0,1]^2),\\
 \int \alpha(x,&y) \D \mu_n(x,y) \\
 &=\sum_{i=0}^{n-1} \int_0^{1/2n} \left\{ \alpha \left( \frac{i}{n} + x, \frac{i}{2n} + x \right) + \alpha \left( \frac{2i+1}{2n} + x , \frac{1}{2} + \frac{1}{2n} + x \right) \right\} \D x.
\end{align*}
We illustrate these measures in figure \ref{defmu}. 
\begin{figure}
\label{defmu}
\centering
\begin{tikzpicture}
\draw[step=0.5cm, gray, very thin, dashed] (0, 0) grid (4, 4);
\draw (-0.15,-0.3) node{$0$};
\draw (2,-0.3) node{$1/2$};
\draw (4,-0.3) node{$1$};
\draw (-0.4,2) node{$1/2$};
\draw (-0.15,4) node{$1$};
\draw[thick, red] (0,0) -- (4,2);
\draw[thick, red] (0,2) -- (4,4);
\end{tikzpicture}
\hspace{1.4cm}
\begin{tikzpicture}
\draw (-0.15,-0.3) node{$0$};
\draw (2,-0.3) node{$1/2$};
\draw (4,-0.3) node{$1$};
\draw (-0.4,2) node{$1/2$};
\draw (-0.15,4) node{$1$};
\draw (4.6,2) node{};
\draw[step=0.5cm, gray, very thin, dashed] (0, 0) grid (4, 4);
\draw[thick, red] (0,0) -- (0.5,0.5);
\draw[thick, red] (1,0.5) -- (1.5,1);
\draw[thick, red] (2,1) -- (2.5,1.5);
\draw[thick, red] (3,1.5) -- (3.5,2);
\draw[thick, red] (0.5,2) -- (1,2.5);
\draw[thick, red] (1.5,2.5) -- (2,3);
\draw[thick, red] (2.5,3) -- (3,3.5);
\draw[thick, red] (3.5,3.5) -- (4,4);
\end{tikzpicture}
\caption{Illustration of $\mu_{\infty}$ (on the left) and of $\mu_n$ (for $n=4$, on the right). The corresponding measures are uniform on the red lines.}
\end{figure}
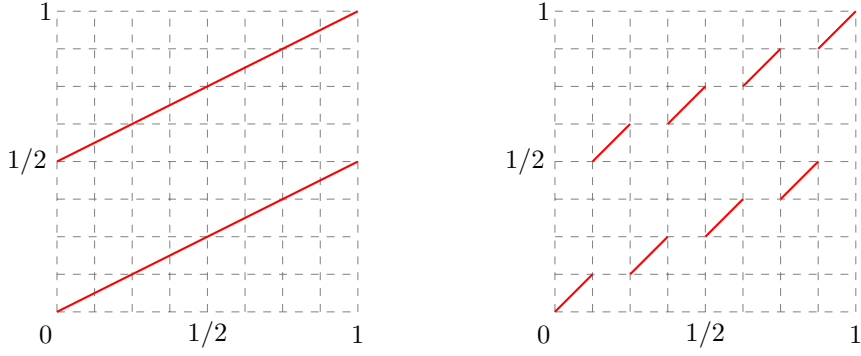
It is clear that in the sense of narrow convergence,
\[
\lim_{n \to \infty}\mu_n = \mu_{\infty}. 
\]
It is also clear that $\overline{\A}(\mu_{\infty}, \mu_{\infty}) = 0$. But now take $n \in \N^*$ and $H$ an admissible flow for $\overline{\Pb}(\mu_{\infty}, \mu_n)$ and take its representation under the form $\lambda \otimes \eta^a$. Chose $i \in \{0, \dots, n-1\}$. For $\lambda$-almost all $a \in [i/n, (2i + 1)/2n]$, the initial marginal $e_0 {} \pf \eta^a$ equals 
\[
\frac{1}{2} \Big( \delta_{a/2} + \delta_{(1+a)/2} \Big),
\]
and the final marginal $e_1 {}\pf \eta^a$ equals
\[
\delta_{a -(i/2n)}.
\]
In particular, one half of the particles with label $a$ start from
\[
\frac{1}{2} + \frac{a}{2} \geq \frac{1}{2} + \frac{i}{2n}
\]
and arrive at
\[
a - \frac{i}{2n} \leq \frac{i}{2n} +\frac{1}{2n}.
\]
In particular, for $\lambda$-almost all $a \in [i/n, (2i + 1)/2n]$,
\begin{equation}
\label{Aetaabig}
\A(\eta^a) \geq \frac{1}{2} \times \frac{1}{2} \left( \frac{1}{2} + \frac{i}{2n} -\frac{i}{2n} - \frac{1}{2n} \right)^2 = \frac{1}{16}\left( 1 - \frac{1}{n} \right)^2.
\end{equation}
For the same reasons, \eqref{Aetaabig} is also valid for $\lambda$-almost all $a \in [(2i + 1)/2n, (i+1)/n]$. As the result does not depend on $i$, in fact \eqref{Aetaabig} is valid for $\lambda$-almost all $a \in [0,1]$, and integrating over $[0,1]$ leads to
\[
\overline{\A}(H) \geq  \frac{1}{16}\left( 1 - \frac{1}{n} \right)^2.
\]
Taking the infimum in the left hand side, we get
\[
\overline{\A}(\mu_{\infty}, \mu_n) \geq  \frac{1}{16}\left( 1 - \frac{1}{n} \right)^2, 
\]
and subsequently,
\[
\liminf_{n \to \infty} \overline{\A}(\mu_{\infty}, \mu_n) \geq \frac{1}{16} > 0 = \overline{\A}(\mu_{\infty}, \mu_{\infty}).
\]
Therefore, as announced, $\overline{\A}(\mu_{\infty}, \bullet)$ is not continuous.
\section{H\"older continuity of the pressure field}
This section is entirely devoted to the proof of Theorem \ref{stabilitypressure}. Once again we start with giving the ideas of the proof. We recall that expression \eqref{pmultlag} lets us interpret the pressure field as the Lagrange multiplier associated to the incompressibility constraint. In other terms, if $(\mathcal{E}, \| \bullet \|_{\mathcal{E}})$ is a space of densities, if $\mu$ is a bistochastic measure, and if we call $p_{\mu}$ the pressure associated to the endpoint conditions $\mu$,
\[
\sup_{\| r \|_{\mathcal{E}}\leq 1} \cg p_{\mu} , r \cd
\]
is seen as the slope of the action $\A(\mu, \bullet)$ in $\mathcal{\mathcal{E}}$ at the point $\lambda$. The theorem states that the slope in $\mathcal{E} = \Div E$ does not depend too much on the endpoints condition. This is a consequence of four estimates. Given $\mu$ and $\nu$ two bistochastic measures, a direction $r\in \mathcal{E}$ and two small parameters $\eps$ and $\delta$, we will see that under certain conditions,
\begin{enumerate}
\item at the endpoints condition $\mu$, the slope in the direction $r$ is bounded from below by a quantity of type
\[
\cg p_{\mu}, r \cd \geq \frac{\A(\mu, \lambda + \delta r) - \A(\mu, \lambda)}{\delta} - M\delta;
\]
\item at the endpoints condition $\nu$, the slope in the direction $r$ is bounded from above by a quantity of type
\[
\cg p_{\nu}, r \cd \leq \frac{\A(\nu, \lambda + \delta r^{\eps}) - \A(\nu, \lambda)}{\delta} + M(\delta + \eps);
\]
\item the number $\A(\nu, \lambda + \delta r^{\eps})$ is not too large with respect to $\A(\mu, \lambda + \delta r)$, as seen in lemma \ref{generalcontinuityaction}:
\[
\A(\nu, \lambda + r^{\eps}) \leq \A(\mu, \lambda + r) + M \bigg( \eps + \big\{1 + N( r^{\eps}) \big\} \frac{\DMK(\mu, \nu) }{\eps^{d+2}} \bigg);
\]
\item the number $\A(\nu, \lambda)$ is not too small with respect to $\A(\mu, \lambda)$:
\[
\A(\nu, \lambda) \geq \A(\mu, \lambda) - M \bigg( \eps + \frac{\DMK(\mu, \nu) }{\eps^{d+2}} \bigg).
\]
\end{enumerate}
Indeed, relying on these four estimates, one can get something like
\[
\cg p_{\nu} - p_{\mu}, r \cd \leq \frac{M}{\delta}\bigg( \eps + \big\{1 + N( r^{\eps}) \big\} \frac{\DMK(\mu, \nu) }{\eps^{d+2}} \bigg) + M(\eps + \delta).
\]
Then we just have to take the good parameters to get the result with variations of density of the form $r = - \Div \xi$. In fact, given a vector field $\xi$, we will not work directly with $r = - \Div \xi$ in the right hand side of the first and second points, and in the third point, but with the density obtained by transporting the Lebesgue measure along the flow induced by $\xi$. This last remark will be explained precisely in the proof.

Let us do it rigorously. We take $\mu$ and $\nu$ two bistochastic measures sufficiently close in a sense to be specified later, $C$ and $M$ as in Lemma \ref{generalcontinuityaction}, a number $ \tau \in (0, 1/4]$ and $\xi = (\xi(t,x) \in \R^d) \in E_{\tau}$ a vector field satisfying $N(\xi) \leq 1$. We can suppose without loss of generality that with these $\tau$ and $M$, \eqref{pl2bv} is valid. We also take two small parameters $\eps$ and $\delta$ that we will fix later on. 

Call $\rho_{\delta}$ the density defined for all $t \in [0,1]$ and all $\alpha \in \mathcal{D}(\T^d)$ by
\[
\int \alpha(z) \rho_{\delta}(t,\D z) = \int \alpha (x + \delta \xi(t,x) ) \D x.
\]
\paragraph{First point.} We chose $\eta_{\mu}$ some solution to $\Pb(\mu)$. We recall that according to the finite diameter property, $\sqrt{\A(\eta_{\mu})} \leq \Diam(d) \leq M$.  For any $\omega \in \C$, we call $T_{\delta}( \omega)$ the curve whose position at time $t \in [0,1]$ is
\[
\omega(t) + \delta \xi(t, \omega(t)).
\]
Then we call $H_{\mu}(\delta) := T_{\delta} {}\pf \eta_{\mu}$. Of course, the density of $H_{\mu}(\delta)$ is $\rho_{\delta}$ and $T_{\delta}$ does not change the endpoints of the trajectories. So $H_{\mu}(\delta)$ is admissible for $\Pb(\mu, \rho_{\delta})$. Moreover, using \eqref{defpressure}, we can estimate the action of $H_{\mu}(\delta)$:
\begin{align*}
\A(\mu, \rho_{\delta}) &\leq \A(H_{\mu}(\delta)) = \frac{1}{2}\int_0^1 \hspace{-5pt} \int \left| \dot{\omega}(t) + \delta \frac{\D}{\D t} \xi(t, \omega(t)) \right|^2 \D \eta_{\mu}(\omega) \D t \\
&= \A(\eta_{\mu}) - \delta \cg p_{\mu} ,\Div \xi \cd + \frac{\delta^2}{2} \int_0^1 \hspace{-5pt} \int \left| \frac{\D}{\D t} \xi(t, \omega(t))\right|^2 \D \eta_{\mu}(\omega) \D t \\
&\leq \A(\eta_{\mu}) - \delta \cg p_{\mu} ,\Div \xi \cd + M \delta^2 N(\xi)^2.
\end{align*}
Finally, as $N(\xi)$ is supposed to be smaller than 1, 
\begin{equation}
\label{pmularge}
\A(\mu, \rho_{\delta}) -\A(\mu, \lambda) + \delta \cg p_{\mu} ,\Div \xi \cd \leq M \delta^2.
\end{equation}
In particular, there exists $\delta_0>0$ such that if $\delta \leq \delta_0$, 
\begin{equation}
\label{estimAmurhodelta}
\sqrt{\A(\mu, \rho_{\delta})} \leq 2 \Diam (d).
\end{equation}
\paragraph{Second point.} For the second estimate, the starting point is \eqref{pmultlag} written for the regularization $\rho_{\delta}^{\eps}$ of $\rho_{\delta}$:
\begin{equation}
\label{pmultlagrhoeps}
\A(\nu, \lambda ) + \cg p_{\nu}, \rho_{\delta}^{\eps} - 1\cd \leq \A(\nu, \rho_{\delta}^{\eps}).
\end{equation}
Then, we remark that as soon as $\delta < 1$, $\det(\Id + \delta \D \xi(t,z))$ is well defined and positive for all $t \in [0,1]$, for $\lambda$-almost all $x \in \T^d$, and
\[
\rho_{\delta}(t, z + \delta \xi(t,z)) = \frac{1}{\det (\Id + \delta \D \xi(t,z))}. 
\]
Moreover, as all the coefficients of $\D \xi(t,z)$ are almost everywhere smaller than one, developing the determinant and still using $\delta < 1$,
\[
\det(\Id + \delta \D \xi(t,z)) = 1 + \delta \Div \xi(t,x) + \delta^2 A(t,z, \delta)
\]
with
\[
\sup_{\delta} \| A(\bullet , \delta) \|_{L^{\infty}_{t,z}} \leq M.
\]
Subsequently, up to taking a smaller $\delta_0>0$, if 
\begin{equation}
\label{deltaleqdelta0}
\delta \leq \delta_0,
\end{equation}
then
\begin{gather}
\| \rho_{\delta} - 1 \|_{L^{\infty}_{t,x}} \leq \frac{1}{4}, \label{rhodeltabig} \\
\notag \|  \rho_{\delta} - 1 + \delta \Div \xi \|_{L^{\infty}_{t,x}} \leq \delta^2 M.
\end{gather}
As a consequence, if we define $\xi^{\eps}$ for all $t \in [0,1]$ and $z \in \T^d$ by
\[
\xi^{\eps}(t,z) := \left\{ \begin{aligned} &0 &&\mbox{if } t \in [0, \eps],\\
& \int \psi^{\eps} (z-x) \xi(s,x) \D x \mbox{ with } s = \frac{t-\eps}{1 - 2 \eps} &&\mbox{if } t \in [\eps, 1-\eps],\\
&0 && \mbox{if } t \in [1-\eps, 1],
\end{aligned}\right.
\]
we get that under condition \eqref{deltaleqdelta0},
\[
\|  \rho_{\delta}^{\eps} - 1 - \delta \Div \xi^{\eps} \|_{L^{\infty}_{t,x}} \leq \delta^2 M.
\]
In addition, if $t \in [0,\tau] \cup [1-\tau, 1]$ and $z \in \T^d$, $\rho_{\delta}^{\eps}(t,z) = 1$ and $\xi^{\eps}(t,z) = 0$. These two remarks are sufficient to make use of the regularity of the pressure field \eqref{pl2bv} proven in \cite{ambfig08}. Indeed,
\begin{align}
|\cg p_{\nu}, \rho_{\delta}^{\eps} - 1 - \delta \Div \xi^{\eps}\cd| &\leq \| p \|_{L^1([\tau, 1-\tau] \times \T^d)}\|  \rho_{\delta}^{\eps} - 1 - \delta \Div \xi^{\eps} \|_{L^{\infty}_{t,x}} \notag \\
&\leq \| p \|_{L^2([\tau, 1-\tau]; BV)}\| M \delta^2  \notag \\
&\leq M \delta^2. \label{rhoeps-1-divxieps}
\end{align}

Now we want to estimate $\xi - \xi^{\eps}$ in $L^2_t L^{\infty}_x$ norm. In the following computations, if $f$ is a function of $t$ and $x$, we will denote by $f(t)$ the function of $x$ $f(t, \bullet)$. First, if $t \in [0,\eps]$, $\xi^{\eps}(t)$ cancels, and
\[
\| \xi(t) - \xi^{\eps}(t) \|^2_{L^{\infty}_x} \leq \left( \int_0^t\|\partial_t \xi(\sigma)\|_{L^{\infty}_x} \D \sigma \right)^2 \leq \eps \int_0^{\eps} \|\partial_t \xi(\sigma)\|^2_{L^{\infty}_x} \D \sigma . 
\] 
Likewise, if $t \in [1-\eps, 1]$,
\[
\| \xi(t) - \xi^{\eps}(t) \|^2_{L^{\infty}_x} \leq \eps \int_{1-\eps}^1\|\partial_t \xi(\sigma)\|^2_{L^{\infty}_x} \D \sigma. 
\] 
Finally, if $t \in [\eps, 1-\eps]$, calling $s = (t-\eps)/(1-2\eps)$,
\begin{align*}
\| \xi(t) - \xi^{\eps}(t) \|^2_{L^{\infty}_x} &= \| \xi(t) - \xi \ast \psi^{\eps}(s) \|^2_{L^{\infty}_x} \\
&\leq 2\| \xi(t) - \xi(s) \|^2_{L^{\infty}_x} + 2\| \xi (s) - \xi \ast \psi^{\eps}(s) \|^2_{L^{\infty}_x}\\
&\leq 2 \left| \int_s^t \| \partial_t \xi(\sigma) \|_{L^{\infty}_x}\D \sigma \right|^2 + M \eps^2\\
&\leq 2\eps \int_{t-\eps}^{t+\eps}  \| \partial_t \xi(\sigma) \|^2_{L^{\infty}_x}\D \sigma  + M \eps^2.
\end{align*}
Gathering these three estimates, we get
\begin{align*}
\int_0^1& \| \xi(t) - \xi^{\eps}(t) \|^2_{L^{\infty}_x} \D t\\
 &\leq M \eps^2 + \eps \int_0^\eps \hspace{-5pt} \int_0^{\eps}  \| \partial_t \xi(\sigma) \|^2_{L^{\infty}_x}\D \sigma \D t + \eps \int_{1-\eps}^1 \int_{1-\eps}^{1}  \| \partial_t \xi(\sigma) \|^2_{L^{\infty}_x}\D \sigma \D t \\
 &\qquad + 2 \eps\int_{\eps}^{1-\eps} \int_{t-\eps}^{t + \eps}  \| \partial_t \xi(\sigma) \|^2_{L^{\infty}_x} \D \sigma  \D t \\
 &\leq M \eps^2 + \eps^2 \int_0^{\eps}\| \partial_t \xi(\sigma) \|^2_{L^{\infty}_x} \D \sigma +  \eps^2 \int_{1-\eps}^{1}\| \partial_t \xi(\sigma) \|^2_{L^{\infty}_x} \D \sigma \\
 &\qquad + 4 \eps^2 \int_0^1 \| \partial_t \xi(\sigma) \|^2_{L^{\infty}_x} \D \sigma \\
 & \leq M \eps^2,
\end{align*}
and so
\[
\| \xi - \xi^{\eps} \|_{L^2_t L^{\infty}_x} \leq M\eps.
\]
Moreover, if $t \in [0,\tau] \cup [1-\tau, 1]$, both $\xi(t)$ and $\xi^{\eps}(t)$ cancel. Consequently, for $M$ large enough,
\begin{align}
| \cg p_{\nu}, \delta \Div \xi^{\eps} - \delta \Div \xi \cd | & = \delta | \cg \nabla p_{\nu}, \xi - \xi^{\eps} \cd | \notag\\
&\leq \delta \| p_{\nu} \|_{L^2([\tau, 1-\tau]; BV)} \| \xi - \xi^{\eps} \|_{L^2_t L^{\infty}_x} \notag\\
& \leq M \delta \eps. \label{divxieps-divxi}
\end{align}
In the end, gathering \eqref{pmultlagrhoeps}, \eqref{rhoeps-1-divxieps} and \eqref{divxieps-divxi}, as soon as \eqref{deltaleqdelta0} holds,
\begin{equation}
\label{pnusmall}
-\A(\nu, \rho_{\delta}^{\eps}) + \A(\nu, \lambda) - \delta \cg p_{\nu}, \Div \xi \cd \leq M \delta ( \delta + \eps).
\end{equation}
\paragraph{Third and fourth point.}
We already saw that if \eqref{deltaleqdelta0} holds, then so do \eqref{rhodeltabig} (and thus $\rho_{\delta} \geq 3/4$) and \eqref{conditionlemma1}. Furthermore under this condition, using \eqref{Nrhoeps}, we get the existence of $K$ only depending on the dimension such that
\[
1 + N\big(\rho_{\delta}^{\eps}\big) \leq \frac{K}{\eps^{d+1}}.
\]
So as soon as \eqref{deltaleqdelta0} holds and if,
\begin{equation}
\label{Nrhodeltaepssmall}
CK \frac{\DMK(\mu, \nu)}{\eps^{(d+1)(d+2)}} \leq \frac{1}{4},
\end{equation}
then Lemma \ref{generalcontinuityaction} applies with $\rho_{\delta}$ (and \textit{a fortiori} also with $\lambda$), that is to say
\begin{gather}
\label{Anurhosmall} \A(\nu, \rho_{\delta}^{\eps}) - \A(\mu, \rho_{\delta}) \leq M \left( \eps + \frac{\DMK(\mu, \nu)}{\eps^{(d+1)(d+2)}}\right), \\
\label{Anulambdabig} \A(\mu, \lambda) - \A(\nu, \lambda) \leq M \left( \eps + \frac{\DMK(\mu, \nu)}{\eps^{(d+1)(d+2)}}\right).
\end{gather}
\paragraph{Conclusion.} The consequence of these points is that under the conditions \eqref{deltaleqdelta0} and \eqref{Nrhodeltaepssmall}, the four inequalities \eqref{pmularge}, \eqref{pnusmall}, \eqref{Anurhosmall} and \eqref{Anulambdabig} hold and summing them and dividing by $\delta$, we obtain (using $\delta \leq \delta_0 < 1$ to bound $\eps$ by $\eps / \delta$)
\[
 \cg p_{\nu} - p_{\mu}, \Div \xi \cd \leq M \left( \delta + \frac{1}{\delta} \left\{ \eps + \frac{\DMK(\mu, \nu)}{ \eps^{(d+1)(d+2)}} \right\} \right).
\]
Now it is straightforward to check that if 
\[
\DMK(\mu, \nu) \leq \min \left( \delta_0^{2 + 2(d+1)(d+2)}, \frac{1}{(4CK)^{1 + (d+1)(d+2)}} \right),
\]
then
\[
\delta := \DMK(\mu, \nu)^{1/[2 + 2(d+1)(d+2)]} \qquad \mbox{and} \qquad \eps := \DMK(\mu, \nu)^{1/[1 + (d+1)(d+2)]}
\]
satisfy \eqref{deltaleqdelta0} and \eqref{Nrhodeltaepssmall} and provide the following inequality
\[
 \cg p_{\mu} - p_{\nu}, \Div \xi \cd \leq M  \DMK(\mu, \nu)^{1/[2+2(d+1)(d+2)]}.
\]
The global H\"older property is deduced from the local one as in the proof of Theorem \ref{continuityaction}.
\qed

\section*{Appendix: proof of Lemma \ref{dacmos}}
We denote by $\DDD$ the geodesic distance on $\T^d$. We call $\Id$ the function that associates to $(t,x)$ the point $x$, $\1$ the function that associates to $(t,x)$ the value $1$. Remark that $N$ is sub-multiplicative: there exists $M$ only depending on the dimension such that for all $a$ and $b$ in $E$, then
\begin{equation}
\label{Nsubmultiplicative}
N(ab) \leq M N(a)N(b).
\end{equation}

We will use the following lemma which is a classical result in the theory of elliptic equations (see for example chapter 3 of \cite{gilbarg2015elliptic}).
\begin{Lem}
\label{solpossionlip}
Let $F$ be a bounded measurable function on the torus whose integral is null and let $U$ be a distributional solution to the Poisson equation
\[
\Delta U = F.
\]
Then $U$ is Lipschitz continuous and there exists $C$ only depending on the dimension such that
\[
\| \nabla U \|_{\infty} \leq C \| F \|_{\infty}.
\]
\end{Lem} 
 \begin{proof}[Proof of lemma \ref{dacmos}]
For given $s \in [0,1]$, $t \in [0,1]$ and $x \in \T^d$, we call
\begin{gather*}
h(t,x) := f(t,x) - g(t,x),\\ 
\rho(s,t,x) := (1-s) f(t,x) + s g(t,x),\\
L := \sup_t \Lip h(t, \bullet),\\
\kappa(t) := \| \partial_t h(t, \bullet) \|_{L^{\infty}(\T^d)},\\
\mathcal{N} := \max\big\{N(f), N(g)\big\}.
\end{gather*}
Notice that for all $s$, $\rho(s,\bullet )$ belongs to $E$ and all its values are greater than $1/2$. Also remark the following estimate 
\[
N\big(\rho(s, \bullet) \big) \leq \mathcal{N}.
\]

Then for all $t$ we look for the unique distributional solution to the Poisson equation
\begin{gather*}
\Delta \theta (t,x) = f(t,x) - g(t,x) = h(t,x),\\
\int \theta(t,x) \D x = 0.
\end{gather*}
(The second equation is only useful to ensure the measurability of $\theta$ with respect to time.) We first analyse the regularity of $\theta$. A direct application of Lemma \ref{solpossionlip} at each time gives
\[
\| \nabla \theta \|_{\infty} \leq C \|h \|_{\infty} \leq C N(h).
\]
Now take $x_1$, $x_2$ in $\T^d$, $t_1$, $t_2$ in $[0,1]$, and call for every $x \in \T^d$
\begin{gather*}
H(x) := h(t_2, x_2 + x) - h(t_1, x_1 + x),\\
\Theta(x) := \theta(t_2, x_2 + x) - \theta(t_1, x_1 + x).
\end{gather*}
Then $\Theta$ is a solution to the Poisson equation
\[
\Delta \Theta = H,
\]
and besides,
\[
\| H \|_{\infty} \leq L \DDD(x_1, x_2) + \int_{t_1}^{t_2} \kappa(t) \D t.
\]
As a consequence, by Lemma \ref{solpossionlip},
\[
\| \nabla \Theta \|_{\infty} \leq C \left( L \DDD(x_1, x_2) + \int_{t_1}^{t_2} \kappa(t) \D t \right).
\]
Using this estimate at $x = 0$, we get
\[
| \nabla \theta (t_2, x_2) - \nabla \theta(t_1, x_1) | \leq C  \left( L \DDD(x_1, x_2) + \int_{t_1}^{t_2} \kappa(t) \D t \right).
\]
This expression is valid for all $t_1$, $t_2$, $x_1$ and $x_2$, which exactly means that $\nabla \theta$ is in $E$, and up to taking a larger $C$, we get
\[
N( \nabla  \theta)  \leq C N(h).
\]

Now define for all $s \in [0,1]$, $t \in [0,1]$ and $x \in \T^d$ the vector
\[
v(s,t,x) := \frac{\nabla \theta(t,x)}{\rho(s,t,x)}. 
\]
Remark that for a fixed $t \in [0,1]$, the following continuity equation holds in $[0,1] \times \T^d$:
\[
\partial_s \rho + \Div (\rho v ) = 0.
\]

Now $\rho\geq 1/2$ and the reciprocal function ($a \mapsto 1/a$) is 4-Lipschitz on $[1/2, \infty[$. So using the estimates that we have on $\rho$ and $v$, \eqref{Nsubmultiplicative}, and taking $M$ large enough,
\begin{align*}
\sup_s N\big(v(s, \bullet)\big) &\leq M N(\nabla \theta) \sup_s N \left( \frac{1}{\rho(s, \bullet)} \right) \\
&\leq M N(\nabla \theta) \sup_s N\Big(\rho(s,  \bullet)\Big) \\
&\leq M N(\nabla \theta) \mathcal{N}.
\end{align*}
In particular, the Cauchy-Lipschitz theorem lets us define the family of flows defined by the ordinary differential equations 
\begin{gather*}
\forall t \in [0,1], \, \forall x \in \T^d, \quad \Phi(0,t,x) = x,\\
\forall s \in [0,1], \,\forall t \in [0,1], \, \forall x \in \T^d, \quad \partial_s \Phi(s,t,x) = v(s,t, \Phi(s,t,x)).
\end{gather*}

It is a classical fact in the theory of the continuity equation (see for example \cite{amb08}) that in this context, the method of characteristic implies that for all $s \in [0,1]$ ans $t \in [0,1]$,
\[
\rho(s,t,x) \D x = \Phi(1,t,\bullet) \pf (\rho(0,t,x) \D x),
\]
and in particular, \eqref{transport} with for all $t \in [0,1]$ and $x \in \T^d$,
\[
\Psi(t,x) := \Phi(1,t,x).
\]

Remark that if $f(t, \bullet) = g(t,\bullet)$, then $\theta(t, \bullet) = 0$, so for all $s \in [0,1]$, $v(s,t,\bullet) = 0$. Assertion \eqref{samedensitynomove} follows easily.

It remains to show the estimate \eqref{psiclosetoid}. But this is a classical result of dependence with respect to the initial conditions and to a parameter of the solutions to an ODE, following the estimates we already have on $v$.
\end{proof}

\paragraph{Acknowledgments.} This work is part of my PhD thesis supervised by Yann Brenier and Daniel Han-Kwan. I would like to thank both of them for the careful reading and advice.
 \bibliography{../../biblio/bibliographie}
 \bibliographystyle{plain}

 \end{document}